\documentclass[11pt,a4paper]{amsart}
\usepackage{latexsym}
\usepackage{amscd}
\usepackage{amsmath}
\usepackage{amssymb}
\usepackage[mathscr]{euscript}
\usepackage{graphicx}
\usepackage{geometry}
\usepackage{mathtools}
\usepackage[all]{xy}
\usepackage[dvipsnames]{xcolor}
\usepackage[colorlinks,final,hyperindex]{hyperref}
\usepackage{tikz}
\usepackage{tikz-cd}
\usetikzlibrary{decorations.pathmorphing,decorations.markings,arrows,calc,shapes.geometric,arrows.meta,positioning}
\usepackage{enumerate}
\usepackage{multirow}

\usepackage[bb=dsserif]{mathalpha}
\usepackage{bm}

\newcommand{\CC}{\ensuremath{\mathbb{C}}}
\newcommand{\RR}{\ensuremath{\mathbb{R}}}
\newcommand{\ZZ}{\ensuremath{\mathbb{Z}}}

\renewcommand{\SS}{\ensuremath{\mathbb{S}}}

\makeatletter

\newcommand{\calT}{\mathcal{T}}

\newcommand{\kk}{\Bbbk}

\newcommand{\scrY}{\mathscr{Y}}
\newcommand{\scrV}{\mathscr{V}}
\newcommand{\scrP}{\mathscr{P}}
\newcommand{\scrQ}{\mathscr{Q}}

\newcommand{\coker}{\mathop{\mathrm{coker}}\nolimits}

\newcommand{\Diop}{\mathop{\mathsf{Dioperads}}\nolimits}
\newcommand{\PDiop}{\mathop{\mathsf{PDioperads}}\nolimits}

\newcommand{\Codiop}{\mathop{\mathsf{Codioperads}}\nolimits}
\newcommand{\PCodiop}{\mathop{\mathsf{PCodioperads}}\nolimits}

\newcommand{\antish}{\ensuremath{\mbox{!`}}}

\tikzset{->-/.style={decoration={
			markings,
			mark=at position #1 with {\arrow{>}}},postaction={decorate}}}
\tikzset{-w-/.style={decoration={
			markings,
			mark=at position #1 with {\arrow{Stealth[fill=white,scale=1.4]}}},postaction={decorate}}}
\tikzset{->-/.default=0.65}
\tikzset{-w-/.default=0.65}
\tikzstyle{bullet}=[circle,fill=black,inner sep=0.5mm]
\tikzstyle{circ}=[circle,draw=black,fill=white,inner sep=0.5mm]
\tikzstyle{vertex}=[circle,draw=black,thick,inner sep=0.5mm]
\tikzstyle{dot}=[draw,circle,fill=black,minimum size=0.5mm,inner sep = 0mm, outer sep = 0mm]
\tikzset{darrow/.style={double distance = 4pt,>={Implies},->},
	darrowthin/.style={double equal sign distance,>={Implies},->},
	tarrow/.style={-,preaction={draw,darrow}},
	qarrow/.style={preaction={draw,darrow,shorten >=0pt},shorten >=1pt,-,double,double
		distance=0.2pt}}

\newcommand{\tikzfig}[1]{\begin{tikzpicture}[auto,baseline={([yshift=-.5ex]current bounding box.center)}]#1\end{tikzpicture}}
\usetikzlibrary{decorations.pathreplacing}

\usepackage{amsthm}
\usepackage{thmtools}
\usepackage[noabbrev,capitalize]{cleveref}
\newtheorem{theorem}{Theorem}
\newtheorem{lemma}[theorem]{Lemma}
\newtheorem{proposition}[theorem]{Proposition}
\newtheorem{corollary}[theorem]{Corollary}

\newtheorem*{theorem*}{Theorem}

\theoremstyle{definition}
\newtheorem{definition}{Definition}

\usepackage[backend=biber, maxbibnames=99, bibencoding=utf8, doi=false, isbn=false, url=false, style=alphabetic]{biblatex}
\title{Koszulity of a certain dioperad}
\author{Alex Takeda}
\begin{document}
\newcommand{\Addresses}{{%
  \bigskip
  \footnotesize
  \noindent A.~Takeda, \textsc{Uppsala University, Department of Mathematics, Box 480, 751 05 Uppsala, Sweden}\par\nopagebreak
  \noindent\textit{E-mail address}: \texttt{alex.takeda@math.uu.se}
 }}

\begin{abstract}
    We establish that the dioperad $\scrY^{(n)}$, encoding bialgebras with a product of degree zero, a coproduct of degree $(1-n)$ and a rank three cyclic tensor, which satisfy a deformed version of the balanced infinitesimal bialgebra condition, is Koszul. This result is established by studying specific subcomplexes of the assocoipahedra of Poirier and Tradler. These subcomplexes are related to a certain type of meromorphic quadratic differential on $\mathbb{CP}^1$, which we call cloven Strebel differentials. Using that geometric interpretation, we can control the topology of the relevant subcomplexes and deduce the vanishing of higher cohomology of the corresponding dioperadic bar complexes.
    \\
    \\
    {\it Mathematics Subject Classification} (2020). 18M85, 30F30; 57K20, 53C12. \\
    \emph{Keywords.} Dioperads, quadratic differentials, Koszul duality.
\end{abstract}

\maketitle

\section{Introduction}
\noindent A \emph{properad} is an object that encodes a type of algebra whose operations have multiple inputs and outputs; one allows composition along any connected directed acyclic graph, and imposes relations that involve compositions given by arbitrary linear combinations of such graphs, of any genus. In a \emph{dioperad}, one only records the data of compositions along trees, and is constrained to impose relations that only involve compositions along trees. A \emph{planar dioperad} is an even more restrictive notion, where one requires each relation to be given by linear combinations of trees that are embedded in the plane in a coherent manner. There are universal enveloping functors
\[ \PDiop \xrightarrow{\Phi} \Diop \xrightarrow{F} \mathsf{Properads}\]
that, starting from a planar dioperad, freely adjoin compositions along any connected graphs, quotienting out by planar genus zero relations. 

Here we are interested in (planar) dioperads that appear in the study of Poincar\'e duality and string topology; one such example is the dioperad $\mathscr{V}^{(n)}$ of \cite{poirier2019koszuality}, which encodes an associative algebra with a compatible copairing of homological degree $-n$. 
The dioperad $\scrY^{(n)}$ is a variant of $\mathscr{V}^{(n)}$ and was defined in \cite{emprin2025properadic}, with the purpose of encoding the data of an orientation on a space $X$ into the chains on the based loop space. This dioperad is defined by planar relations, in other words, we have
\[ \mathscr{Y}^{(n)} = \Phi(\mathscr{Y}^{(n)}_{pl})\]
for a certain planar dioperad $\mathscr{Y}^{(n)}_{pl}$, so algebras over $\mathscr{Y}^{(n)}$ are defined by planar relations: such an algebra $A$ has
\begin{itemize}
    \item A product $\mu \colon A \otimes A \to A$ of degree zero,
    \item A coproduct $\alpha \colon A \to A \otimes A$ of homological degree $(1-n)$,
    \item A cyclically symmetric rank three tensor $\beta \in A \otimes A \otimes A$ of homological degree $(2-2n)$,
\end{itemize} 
satisfying the following relations.

\small\begin{align*}
    &\tikzfig{
        \node [vertex] (bot) at (0,-0.2) {$\mu$};
        \node [vertex] (top) at (-0.5,0.5) {$\mu$};
        \draw [->-] (-1,1) to (top);
        \draw [->-] (0,1) to (top);
        \draw [->-] (top) to (bot);
        \draw [->-] (1,1) to (bot);
        \draw [->-] (bot) to (0,-1);
    } ~ - ~ \tikzfig{
        \node [vertex] (bot) at (0,-0.2) {$\mu$};
        \node [vertex] (top) at (0.5,0.5) {$\mu$};
        \draw [->-] (-1,1) to (bot);
        \draw [->-] (0,1) to (top);
        \draw [->-] (top) to (bot);
        \draw [->-] (1,1) to (top);
        \draw [->-] (bot) to (0,-1);
    } ~ \quad ; \quad \tikzfig{
        \node [vertex] (top) at (0,0.5) {$\mu$};
        \node [vertex] (bot) at (0,-0.5) {$\alpha$};
        \draw [->-] (-1,1) to (top);
        \draw [->-] (top) to (bot);
        \draw [->-] (1,1) to (top);
        \draw [->-] (bot) to (-1,-1);
        \draw [->-] (bot) to (1,-1);
    } ~ - ~ \tikzfig{
        \node [vertex] (top) at (0.5,0) {$\alpha$};
        \node [vertex] (bot) at (-0.5,0) {$\mu$};
        \draw [->-] (-1,1) to (bot);
        \draw [->-] (top) to (bot);
        \draw [->-] (1,1) to (top);
        \draw [->-] (bot) to (-1,-1);
        \draw [->-] (top) to (1,-1);
    } ~ - ~ \tikzfig{
        \node [vertex] (top) at (-0.5,0) {$\alpha$};
        \node [vertex] (bot) at (0.5,0) {$\mu$};
        \draw [->-] (-1,1) to (top);
        \draw [->-] (top) to (bot);
        \draw [->-] (1,1) to (bot);
        \draw [->-] (top) to (-1,-1);
        \draw [->-] (bot) to (1,-1);
    } ~;\\
    &\tikzfig{
        \node [vertex] (left) at (-0.5,0) {$\alpha$};
        \node [vertex] (right) at (0.5,0) {$\mu$};
        \draw [->-] (0,1) to (left);
        \draw [->-] (0,-1) to (right);
        \draw [->-] (left) to (-1.2,0);
        \draw [->-] (left) to (right);
        \draw [->-] (right) to (1.2,0);
    } ~ - ~ \tikzfig{
        \node [vertex] (left) at (-0.5,0) {$\mu$};
        \node [vertex] (right) at (0.5,0) {$\alpha$};
        \draw [->-] (0,1) to (right);
        \draw [->-] (0,-1) to (left);
        \draw [->-] (left) to (-1.2,0);
        \draw [->-] (right) to (left);
        \draw [->-] (right) to (1.2,0);
    } ~ + (-1)^n \left(\tikzfig{
        \node [vertex] (left) at (-0.5,0) {$\alpha$};
        \node [vertex] (right) at (0.5,0) {$\mu$};
        \draw [->-] (0,1) to (right);
        \draw [->-] (0,-1) to (left);
        \draw [->-] (left) to (-1.2,0);
        \draw [->-] (left) to (right);
        \draw [->-] (right) to (1.2,0);
    } ~ - ~ \tikzfig{
        \node [vertex] (left) at (-0.5,0) {$\mu$};
        \node [vertex] (right) at (0.5,0) {$\alpha$};
        \draw [->-] (0,1) to (left);
        \draw [->-] (0,-1) to (right);
        \draw [->-] (left) to (-1.2,0);
        \draw [->-] (right) to (left);
        \draw [->-] (right) to (1.2,0);
    }\right) ~;
\end{align*}
\begin{align*}
    &\tikzfig{
        \node [vertex] (top) at (0,0.4) {$\mu$};
        \node [vertex] (bot) at (0,-0.6) {$\beta$};
        \draw [->-] (bot) to (top);
        \draw [->-] (bot) to (-1.2,-1);
        \draw [->-] (bot) to (1.2,-1);
        \draw [->-] (-1.2,0.4) to (top);
        \draw [->-] (top) to (0,1);
    } ~ - \qquad \tikzfig{
        \node [vertex] (top) at (0,0.1) {$\beta$};
        \node [vertex] (bot) at (-0.7,-0.5) {$\mu$};
        \draw [->-] (top) to (0,1);
        \draw [->-] (bot) to (-1.2,-1);
        \draw [->-] (top) to (1.2,-1);
        \draw [->-] (-1.2,0.4) to (bot);
        \draw [->-] (top) to (bot);
    } ~ + ~ \tikzfig{
        \node [vertex] (top) at (0,0.4) {$\alpha$};
        \node [vertex] (bot) at (0,-0.6) {$\alpha$};
        \draw [->-] (top) to (bot);
        \draw [->-] (bot) to (-1.2,-1);
        \draw [->-] (bot) to (1.2,-1);
        \draw [->-] (-1.2,0.4) to (top);
        \draw [->-] (top) to (0,1);
    } ~ - ~ \tikzfig{
        \node [vertex] (top) at (0,0.1) {$\alpha$};
        \node [vertex] (bot) at (-0.7,-0.5) {$\alpha$};
        \draw [->-] (top) to (0,1);
        \draw [->-] (bot) to (-1.2,-1);
        \draw [->-] (top) to (1.2,-1);
        \draw [->-] (-1.2,0.4) to (bot);
        \draw [->-] (bot) to (top);
    } ~; \\ \\
    &\tikzfig{
        \node [vertex] (top) at (0,0.5) {$\alpha$};
        \node [vertex] (bot) at (0,-0.5) {$\beta$};
        \draw [->-] (bot) to (top);
        \draw [->-] (top) to (-1,1);
        \draw [->-] (top) to (1,1);
        \draw [->-] (bot) to (-1,-1);
        \draw [->-] (bot) to (1,-1);
    } ~ - ~ \tikzfig{
        \node [vertex] (left) at (-0.5,0) {$\alpha$};
        \node [vertex] (right) at (0.5,0) {$\beta$};
        \draw [->-] (right) to (left);
        \draw [->-] (left) to (-1,1);
        \draw [->-] (left) to (-1,-1);
        \draw [->-] (right) to (1,-1);
        \draw [->-] (right) to (1,1);
    } ~ + (-1)^n \left( ~ \tikzfig{
        \node [vertex] (top) at (0,0.5) {$\beta$};
        \node [vertex] (bot) at (0,-0.5) {$\alpha$};
        \draw [->-] (top) to (bot);
        \draw [->-] (top) to (-1,1);
        \draw [->-] (top) to (1,1);
        \draw [->-] (bot) to (-1,-1);
        \draw [->-] (bot) to (1,-1);
    } ~ - ~ \tikzfig{
        \node [vertex] (left) at (-0.5,0) {$\beta$};
        \node [vertex] (right) at (0.5,0) {$\alpha$};
        \draw [->-] (right) to (left);
        \draw [->-] (left) to (-1,1);
        \draw [->-] (left) to (-1,-1);
        \draw [->-] (right) to (1,-1);
        \draw [->-] (right) to (1,1);
    } ~ \right)
\end{align*}\normalsize
Note that all of these relations, and the cyclicity condition on $\beta$, can be given by linear combinations of diagrams embedded in a disc with a consistent labeling of inputs and outputs, using only the cyclic action, which comes from the rotational symmetry of the disc; this is why this dioperad is planar, as we will explain in \cref{sec:planar}.
In this note the following theorem is established:
\begin{theorem}\label{thm:main}
    The dioperad $\mathscr{Y}^{(n)}$ is Koszul.
\end{theorem}
\noindent We will prove this result by giving a geometric interpretation of the bar complexes associated to (the Koszul dual of) the planar dioperad $\scrY^{(n)}_{pl}$. This interpretation is in the same spirit as the description of the ``PROP of open-closed marked surfaces'' of \cite{kontsevich2025pre} using moduli spaces of meromorphic Strebel differentials. In \cref{sec:planar} we will recall this theory, in the specific case of differentials on $\mathbb{CP}^1$ with a higher-order pole at infinity, or equivalently, polynomial quadratic differentials on $\CC$. The moduli spaces of such objects have a natural non-compact cell complex structure that is dual to the cell complex structure of the \emph{assocoipahedra} of \cite{poirier2018combinatorics}, and gives an alternative proof of their compatibility, established in that reference; by the results of \cite{poirier2019koszuality} this implies Koszulity of the dioperad $\mathscr{V}^{(n)}$.

In the $\mathscr{Y}^{(n)}$ case, one needs to remove a certain subcomplex of \emph{cloven Strebel differentials}; these are quadratic differentials that split the complex plane into regions in a specific way. Using geometric arguments, we can control the topology of this subcomplex, concluding that it is always homologous to a bouquet of $(k-2)$-dimensional spheres, where $k$ is the number of outputs we are looking at. Together with a comparison between the bar complexes associated to $\mathscr{V}^{(n)}$ and $\mathscr{Y}^{(n)}$, vanishing of this homology outside of that specific degree implies vanishing of higher cohomology of the relevant bar complex, establishing our main theorem.

\subsection{Relation to other results}
We would like to discuss the relation between the result of this note and known related statements about properadic and dioperadic Koszulity. In general, Koszulity of a quadratic dioperad $\scrQ$ does not imply Koszulity of its enveloping properad $F\scrQ$, and even when both of these hold, the corresponding cofibrant resolutions are not directly obtained from one another \cite{merkulov2009deformation}. However, there is a class of examples for which these notions of Koszulity are related. We will use terminology from the recent preprint \cite{khoroshkin2026properad}: one says that a quadratic dioperad $\scrP$ is \emph{genus-contractible} when the quadratic dual of $F\scrP$ is in fact a codioperad, that is, its decomposition map does not generate any higher genus terms.

If $\scrQ$ is genus-contractible and moreover $F\mathscr{Q}$ is Koszul \textit{as a properad}, then $\scrQ$ is a Koszul dioperad, and $F\mathscr{Q}$ has a cofibrant resolution of a particularly nice form
\[ \Omega^\diamond_\mathrm{prop}(\mathscr{Q}^{\antish}) \to F\mathscr{Q}, \]
where the functor $\Omega^\diamond_\mathrm{prop}(-)$ gives the quasi-free coproperad with zero higher-genus decomposition maps \cite[Propositions~48,49]{merkulov2009deformation}. Merkulov and Vallette refer to such properads as \emph{Koszul-contractible}. In these cases, the notion of ``$F\mathscr{Q}$-algebras up to homotopy'' can be parametrized purely by genus zero data. 

Let us point out three other cases, closely related to $\scrY^{(n)}$, for which Koszulity has been studied. The first one is the already-mentioned dioperad $\scrV^{(n)}$ of Poirier and Tradler, following earlier work of Tradler and Zeinalian \cite{tradler2007infinity}; algebras over this dioperad are pre-Calabi--Yau algebras, in the formulation of \cite{kontsevich2025pre}.
\begin{theorem*}[\cite{poirier2019koszuality,leray2025precalabiyau}]
    The dioperad $\scrV^{(n)}$ is Koszul but not genus-contractible.
\end{theorem*}

\noindent Another one is the dioperad of balanced infinitesimal bialgebras $\mathrm{BIB}^\lambda$ studied by Quesney in \cite{quesney2024balanced}; the algebras that this dioperad encodes are associative analogues of Lie bialgebras, and have a coproduct of degree $\lambda$. It was proved in that paper that this dioperad is genus-contractible, but using deformation theory techniques, Merkulov showed in \cite{merkulov2025complex} the following result.
\begin{theorem*}
    The properad $F(\mathrm{BIB}^\lambda)$ is not Koszul.
\end{theorem*}

\noindent Finally, the third dioperad that we would like to mention is the dioperad $\mathrm{DPois}$ encoding the double Poisson algebras of van den Bergh; this dioperad was studied by Leray in \cite{leray2019protoperadsI,leray2020protoperadsII}, who together with Vallette in \cite{leray2025precalabiyau} showed the following result.
\begin{theorem*}
    The dioperad $\mathrm{DPois}$ is Koszul-contractible.
\end{theorem*}

\noindent The relationship between $\scrY^{(n)}$ and these three other dioperads was established in \cite{leray2025precalabiyau,quesney2024balanced,emprin2025properadic}. Their Koszul dual dioperads are related by inclusions
\[ \mathrm{DPois}^! \overset{n=2}\hookrightarrow (\mathrm{BIB}^{1-n})^! \hookrightarrow (\scrY^{(n)})^! \hookrightarrow (\scrV^{(n)})^! \]
where the superscript on the first arrow means that this map only exists for that value of $n$.\footnote{There are shifted versions of double Poisson algebras, but the results we cite are only phrased for the unshifted case.} We note that the first and third maps do not respect weight-grading, and do not come from taking quadratic dual maps. Taking linear dual, we have quotients of codioperads
\[ \mathrm{DPois}^{\antish} \overset{n=2}\twoheadleftarrow (\mathrm{BIB}^{1-n})^{\antish} \twoheadleftarrow (\scrY^{(n)})^{\antish} \twoheadleftarrow (\scrV^{(n)})^{\antish} \]
and taking cobar, quotients of dg codioperads
\[ \mathrm{DPois}_\infty \overset{n=2}\twoheadleftarrow \Omega (\mathrm{BIB}^{1-n})^{\antish} \twoheadleftarrow \scrY^{(n)}_\infty \twoheadleftarrow \scrV^{(n)}_\infty = \mathrm{pCY}, \]
so the notion of a  $\scrY^{(n)}_\infty$-algebra is extremely close to the notion of a $\scrV^{(n)}_\infty$ or pre-CY algebra; in fact, it is just a pre-CY algebra with vanishing copairing. 
We summarize the results cited above in a table for convenience.
\begin{center}
    \begin{tabular}{|c||c|c|c|c|}
    \hline
                            & $\mathrm{DPois}$ & $\mathrm{BIB}^{\lambda}$  & $\scrY^{(n)}$ & $\scrV^{(n)}$ \\
    \hline
    Dioperad is Koszul?     & Yes              & Unknown                   & Yes           & Yes \\
    \hline
    Properad is Koszul?     & Yes              & No                    & Unknown           & Unknown \\
    \hline
    \end{tabular}
\end{center}

\noindent \textbf{Acknowledgments.} I would like to thank C.~Emprin, E.~Getzler, S.~Merkulov and B.~Vallette for helpful discussions, and especially A.~Khoroshkin and V.~Dotsenko for pointing out an important error in a previous version. This work was supported by the Knut and Alice Wallenberg Foundation.

\section*{Notation and conventions}
\noindent In this note we work over a field $\kk$ of characteristic zero, and all our complexes and algebras are in $\kk$-vector spaces.

\section{Planar dioperads}\label{sec:planar}
\noindent The dioperads that we will consider are all planar in a specific sense; this notion appeared in the work of Ward \cite{ward2019six}. In a loose sense, these are dioperads whose operations are parametrized by directed trees with a planar structure, that is, with the data of the embedding into a disk, modulo isotopy. Let us define this notion a bit more precisely.

For each $k \ge 1$ (outgoing arity) and non-negative integers $i_1,\dots,i_k$ (incoming arities), 
we consider the embeddings
\[ C_k \hookrightarrow \SS_k, \quad C_k \hookrightarrow (\SS_{i_1+\dots+i_k})^{op} \]
of the cyclic group of order $k$ into the symmetric group on $k$ letters, by cyclic permutation, and by inverse permutation of the blocks of $i_1,\dots,i_k$ letters. We think of this $C_k$ action as rotating a disk with $k$ outputs and $i_1 + \dots + i_k$ inputs.
\begin{definition}
    A planar dioperad is a system of $R$-modules $\{\scrP{(k;i_1,\dots,i_k)}\}_{k \ge 1, i_a \ge 0}$, 
    together with an action of $C_k$ on
    \[ \bigoplus_{\{i_a\}} \scrP{(k;i_1,\dots,i_k)} \]
    such that, for any fixed tuple $(i_1,\dots,i_k)$, the generator $e^{2\pi i/k}$ sends 
    \[ \scrP{(k;i_1,\dots,i_k)} \to \scrP{(k;i_2,\dots,i_k,i_1)}, \]
    a unit $R \to \scrP_{(1;1)}$, and a composition map
    \begin{align*} 
        \gamma_{a,b;c} \colon &\scrP{(l;j_1,\dots,j_l)} \otimes \scrP{(k;i_1,\dots,i_k)} \to \\
        &\scrP(k+l-1;i_1,\dots,i_{a-1} + j_b - c,j_{b+1},\dots,j_l,j_1,\dots,j_{b-1},c+ i_a, i_{a+1},\dots,i_k).
    \end{align*}
    The unit and composition maps must intertwine all the $C_k$-actions appropriately, and must satisfy unit and associativity axioms.
\end{definition}

\noindent We avoid writing down with indices the `appropriate' compatibility condition with respect to the cyclic action since the formula is not enlightening, but similar to the usual axioms for a dioperad. Note that the arities $(1;i)$ part of the dioperad form a planar (aka non-$\Sigma$) operad.  

Plainly speaking, to present a planar dioperad $\scrP$, we give an $R$-module of generating operations of arity $(k;i_1 + \dots + i_k)$ for each tuple $(k;i_1,\dots,i_k)$, corresponding to a disc with $k$ outputs and $i_a$ inputs in between each output, and when giving a relation between such operations, we are allowed to use only the cyclic rotations of these diagrams. Note that this is different from the situation of an ordinary dioperad, where we are allowed to use any permutation of inputs and outputs when writing the relations. In this sense, planar dioperads are less general than dioperads, in the same way that non-$\Sigma$ operads are less general than operads. Let us denote by $\PDiop$ the category of planar dioperads. There is a forgetful functor
\[ \Diop \to \PDiop \]
from the category of dioperads, which forgets the $\SS_k \times \SS_{i_1+\dots+i_k}^{op}$ action on $\scrP_{(k;i_1,\dots,i_k)}$ down to a $C_k$ action by the embedding we specified above. This functor has a left adjoint
\[ \Phi \colon \PDiop \to \Diop \]
given by induction of representations. For any planar dioperad $\scrP = \{\scrP_{(k;i_1,\dots,i_k)}\}$, its image under $\Phi$ has components
\[ \Phi \scrP_{(k;i_1,\dots,i_k)} = \scrP_{(k;i_1,\dots,i_k)} \underset{R[C_k]}{\otimes} R\left[\SS_k \times  \SS_{i_1+\dots+i_k}^{op}\right] \] %

\subsection{Koszul duality}
All the notions of quadratic dioperad, its dioperadic Koszul dual, bar-cobar complexes have immediate generalizations to the world of planar dioperads. In particular, any quadratic planar dioperad $\scrP$ has a Koszul dual planar dioperad $\scrP^!$ and Koszul dual planar codioperad $\scrP^{\antish}$, and we have a canonical isomorphism of quadratic dioperads
\[ (\Phi \scrP)^! \cong \Phi (\scrP^!) \]
and a canonical isomorphism of dg codioperads
\[ \mathbf{B}(\Phi \scrP) \cong \Phi_\mathrm{co}(\mathbf{B}_{pl} \scrP) \]
where $\Phi_\mathrm{co}$ denotes the left adjoint to the forgetful functor $\Codiop \to \PCodiop$. 

Just like the bar construction on dioperads (or operads), the bar construction on planar dioperads has a \emph{syzygy} grading, such that the differential has degree $+1$; we denote this grading by a subscript to $\mathbf{B}$. In each arity $(k;i_1,\dots,i_k)$, the isomorphism above implies that we have an \emph{isomorphism} of chain complexes
\[\mathbf{B}^*(\Phi \scrP)_{(k;i_1,\dots,i_k)} \cong (\mathbf{B}^*_{pl} \scrP)_{(k;i_1,\dots,i_k)} \otimes R\left[\textstyle{\prod}_a \SS_{i_a}\backslash \SS_{\sum_a i_a} \right] \]
For any quadratic dioperad $\mathscr{Q}$, there is a canonical isomorphism of dg codioperads
\[ \mathscr{Q}^{\antish} \xrightarrow{\cong} \ker(\mathbf{B}^0 \mathscr{Q} \to \mathbf{B}^1 \mathscr{Q}) = H^0(\mathbf{B}\mathscr{Q}) \]
which in general may fail to be an isomorphism, since $H^*(\mathbf{B}\mathscr{Q})$ could be nontrivial in positive syzygy. Recall now the \emph{Koszul criterion} for Koszulity of a dioperad. One says that the quadratic dioperad $Q$ is \emph{Koszul} when the two following equivalent conditions hold:
\begin{enumerate}[i.]
    \item The canonical map $\Omega(\mathscr{Q}^!) \to \mathscr{Q}$ is a quasi-isomorphism of dg dioperads, and
    \item The canonical map $\mathscr{Q}^{\antish} \to \mathbf{B} \mathscr{Q}$ is a quasi-isomorphism of dg codioperads, which happens if and only if $H^*(\mathbf{B}\mathscr{Q})$ is concentrated in degree zero.
\end{enumerate}
Putting this all together, one can check Koszulity of dioperads of the form $\scrQ = \Phi \scrP$ while staying in the world of planar co/dioperads and their bar complexes.
\begin{proposition}\label{prop:KoszulDiop}
    Let $\scrP$ be a planar dioperad. Then $\Phi \scrP$ is a Koszul dioperad if and only if for each arity $(k;i_1,\dots,i_k)$ the complex $(\mathbf{B}^*_{pl} \scrP)_{(k;i_1,\dots,i_k)}$ has cohomology concentrated in degree zero.
\end{proposition}

\noindent By definition, the complex $(\mathbf{B}^*_{pl} \scrP)_{(k;i_1,\dots,i_k)}$ has generators indexed by directed trees embedded in the disk modulo isotopies, where each vertex has a distinguished outgoing arrow (marking the first output) and is labeled by an element of the appropriate space $\scrP_{(k;i_1,\dots,i_k)}$. Note that since we excluded outgoing arity zero, we must not have any sinks in this directed tree. The differential is given by summing over \emph{contractions} of internal edges where we apply the appropriate composition map.

\section{Koszulity of the dioperad $\scrV^{(n)}$}\label{sec:koszulity}
\noindent We now recall the result of Poirier and Tradler establishing the Koszulity of the dioperad $\scrV^{(n)}$, and give it a more geometric interpretation. This is a quadratic dioperad defined as
\[ \scrV^{(n)} \coloneqq \frac{\calT\left(\mu =
    \begin{tikzpicture}[baseline=0ex,scale=0.2]
        \draw (0,2) node[above] {$\scriptstyle{1}$}
        -- (2,1);
        \draw (4,2) node[above] {$\scriptstyle{2}$}
        -- (2,1) -- (2,0) node[below] {$\scriptstyle{1}$};
        \draw[fill=white] (2,1) circle (10pt);
    \end{tikzpicture}
    ~ ; ~ \nu = 
    \begin{tikzpicture}[baseline=0ex,scale=0.2]
        \draw (0,0)  node[below] {$\scriptstyle{1}$} -- (0,0.5);
        \draw (2,0) node[below] {$\scriptstyle{2}$} -- (2,0.5);
        \draw[fill=white] (-0.3,0.5) rectangle (2.3,1);
    \end{tikzpicture} ~ = ~
    \begin{tikzpicture}[baseline=0ex,scale=0.2]
        \draw (0,0)  node[below] {$\scriptstyle{2}$} -- (0,0.5);
        \draw (2,0) node[below] {$\scriptstyle{1}$} -- (2,0.5);
        \draw[fill=white] (-0.3,0.5) rectangle (2.3,1);
    \end{tikzpicture}
    \right)}
    {\left(\begin{tikzpicture}[baseline=0.5ex,scale=0.2]
        \draw (0,4) node[above] {$\scriptstyle{1}$} --(4,0)
        -- (4,-1) node[below] {$\scriptstyle{1}$};
        \draw (4,4) node[above] {$\scriptstyle{2}$} -- (2,2);
        \draw (8,4) node[above] {$\scriptstyle{3}$}-- (4,0);
        \draw[fill=white] (2,2) circle (10pt);
        \draw[fill=white] (4,0) circle (10pt);
    \end{tikzpicture}
    -
    \begin{tikzpicture}[baseline=0.5ex,scale=0.2]
        \draw (0,4) node[above] {$\scriptstyle{1}$}
        -- (4,0) -- (4,-1) node[below] {$\scriptstyle{1}$};
        \draw (4,4) node[above] {$\scriptstyle{2}$} -- (6,2);
        \draw (8,4) node[above] {$\scriptstyle{3}$} -- (4,0);
        \draw[fill=white] (6,2) circle (10pt);
        \draw[fill=white] (4,0) circle (10pt);
    \end{tikzpicture}
    ~;~
    \begin{tikzpicture}[scale=0.2,baseline=1ex]
        \draw (0,4) node[above] {$\scriptstyle{1}$} -- (0,2) -- (1,0.5) -- (1,-0.5) node[below] {$\scriptstyle{1}$};
        \draw (4,-0.5) node[below] {$\scriptstyle{2}$} -- (4,2) ;
        \draw (2,2) -- (1,0.5);
        \draw[fill=white] (1,0.5) circle (10pt);
        \draw[fill=white] (1.7,2) rectangle (4.3,2.5);
    \end{tikzpicture}
    ~ - ~
    \begin{tikzpicture}[scale=0.2,baseline=1ex]
        \draw (0,4) node[above] {$\scriptstyle{1}$} -- (0,2) -- (-1,0.5) -- (-1,-0.5) node[below] {$\scriptstyle{2}$};
        \draw (-4,-0.5) node[below] {$\scriptstyle{1}$} -- (-4,2) ;
        \draw (-2,2) -- (-1,0.5);
        \draw[fill=white] (-1,0.5) circle (10pt);
        \draw[fill=white] (-1.7,2) rectangle (-4.3,2.5);
    \end{tikzpicture} \right)} \]
where $\mu$ is in homological degree zero and $\nu$ in homological degree $-n$, and the symbol $\calT$ indicates the free dioperad generated by those elements; this is indexed over directed trees. In other words, a $\scrV^{(n)}$-algebra is a graded associative algebra $(A,\mu)$ with a symmetric copairing $\nu$ of degree $-n$, satisfying
\[ \mu(x,\nu') \otimes \nu'' = (-1)^{n|x|} \nu' \otimes \mu(\nu'',x) \]
for all $x \in A$.

This is a planar dioperad, in the sense that it is in the image of the functor $\Phi$. To see this, note that the symmetry condition of $\nu$ simply says that it lives in the trivial representation of $C_2$, and in the two generating relations, we only used cyclic permutations, in fact, only the trivial permutation. So if we define
\[ \scrV^{(n)}_{pl} = \calT_{pl}(\dots)/(\dots) \]
with the same diagrams as above, we have a planar dioperad such that $\scrV^{(n)} \cong \Phi(\scrV^{(n)}_{pl})$. Poirier and Tradler's proof of the Koszulity of $\scrV^{(n)}$ relies on the identification of the complexes $(\mathbf{B}^*_{pl} P)_{(k;i_1,\dots,i_k)}$ with the cellular chain complex of certain polytopal complexes called \emph{assocoipahedra}. We rephrase their result in \cite{poirier2019koszuality} in the language of planar dioperads.
\begin{theorem}
    There is an isomorphism
    \[ (\mathbf{B}^*_{pl} (\scrV^{(n)}_{pl})^!)_{(k;i_1,\dots,i_k)} \cong C_*(Z_{(k;i_1,\dots,i_k)}) \]
    between the bar complex of the planar Koszul dual of the $\scrV^{(n)}$-dioperad and the cellular chain complex of an assocoipahedron.
\end{theorem}
\noindent These polytopal complexes were earlier proven by the same authors to be contractible, and later realized as convex polytopes by Pilaud in \cite{pilaud2022pebble}.

\subsection{Polynomial quadratic differentials}
We will now describe how these complexes relate to moduli spaces of meromorphic Strebel differentials on $\mathbb{CP}^1$ with one higher-order pole. This relation fits into a larger picture of moduli spaces of meromorphic differentials on curves of all genera, that was worked out in \cite{kontsevich2025pre}. The case of genus zero curves with one or two punctures, however, does not fit exactly in what is written in that reference, since it relies on a theorem of \cite{gupta2019meromorphic} which requires
\[ 2g-2-\text{number of punctures} <0. \]
Therefore we need to use another description in our case; here we use instead \cite{dias2021quadratic}, following earlier results of \cite{au2006prescribed}.

\noindent Let us fix $N \ge 2$ and let $Q(N+2)$ denote the space of monic centered complex polynomials of degree $(N-2)$; centered indicates that the sum of its zeros vanishes. We identify these polynomials with quadratic differentials by setting
\[ \varphi(z) = f(z) dz^2, \]
which is a meromorphic quadratic differential on $\mathbb{CP}^1$ with a pole of order $N+2$ at infinity; the difference of $4$ from the degree of $f$ comes from the $dz^2$ factor by changing coordinates $z \mapsto 1/z$. For large $|z|$ we can approximate
\[ \varphi(z) \sim z^{N-2} dz^2 \]
which is a quadratic differential that has a single zero at the origin, with $N$ \emph{critical leaves} going to infinity at angles $2\pi k/N$; this is the asymptotic behavior of the critical leaves of $\varphi(z)$.

To each quadratic differential $\varphi$ on some complex curve $\Sigma$ we can assign two \emph{measured foliations}; these are foliations endowed with transverse measures. We assign to $\varphi$ the \emph{horizontal measured foliation} and the \emph{vertical measured foliation}. The leaves of these foliations are defined by requiring the tangent vector $\vec{v}$ to satisfy
\[ \varphi(\vec{v},\vec{v}) > 0 \text{\ and\ } \varphi(\vec{v},\vec{v}) < 0,\]
respectively; note that for the constant quadratic differential $dz^2$ on $\mathbb{C}$ these are the horizontal and vertical lines on the complex plane. 

For any fixed genus and prescribed pole data, there is a corresponding topological space of measured foliations on a punctured topological surface of genus $\Sigma$ with prescribed behavior at infinity; in our case let us denote by $\mathrm{MF}(N+2)$ the measured foliations on the 2-sphere that have a pole of order $(N+2)$ at infinity. We then have two maps
\[ Q(N+2) \overunderset{p_h}{p_v}{\rightrightarrows} \mathrm{MF}(N+2), \]
assigning to a quadratic differential its horizontal and vertical foliations. Moreover, given any measured foliation, we get a metric on its leaf space, which is a tree; this gives an isomorphism 
\[ \mathrm{MF}(N+2) \xrightarrow{\cong} \mathrm{MetTr}(N) \]
to the space of metric trees with $N$ infinite-length edges going to infinity; this space is homeomorphic to $\mathbb{R}^{N-3}$.
\begin{theorem}\label{thm:homeo}
    For any $N \ge 3$, the map
    \[ (p_h, p_v) \colon Q(N+2) \to \mathrm{MF}(N+2) \times  \mathrm{MF}(N+2) \]
    that assigns to a quadratic differential its horizontal and vertical foliations is a homeomorphism.
\end{theorem}

\noindent We say that a quadratic differential is \emph{Strebel} if its horizontal foliation has measure zero. The theorems above then give identifications
\[ Q^\mathrm{Str}(N+2) = p_h^{-1}(0) \cong \mathrm{MF}(N+2) \cong \mathrm{MetTr}(N) \cong \mathbb{R}^{N-3}, \]
between a monic centered polynomial $f(z)$ whose quadratic differential $f(z)dz^2$ is Strebel, the vertical foliation of that quadratic differential and the metric tree given by its leaf space.

\subsection{Regularized Strebel differentials}
Each metric tree in $\mathrm{MetTr}(N)$ has $N$ leaves of infinite length. We will now divide these $N$ leaves into incoming and outgoing leaves, and \emph{regularize} the outgoing leaves to have finite length. Let $(k;i_1,\dots,i_k)$ be a tuple of integers where $k \ge 2$ and each $i_a \ge 0$.
\begin{definition}
    The regularized moduli space of Strebel differentials associated to the tuple $(k;i_1,\dots,i_k)$ is the space 
    \[ Q^\mathrm{Str}_\mathrm{reg}(k;i_1,\dots,i_k) = Q^\mathrm{Str}(k + i_1 + \dots + i_k + 2) \times (\mathbb{R}_{>0})^k/\mathbb{R}_{>0}, \]
    when $k+ i_1 + \dots + i_k \ge 3$. Here the monoid $\mathbb{R}_{>0}$ acts on $(\mathbb{R}_{>0})^k$ by addition to all coordinates. We set $Q^\mathrm{Str}_\mathrm{reg}(2;0,0) = \{*\}$ for consistency.
\end{definition}
\noindent To interpret each point in the regularized moduli space, given 
\[ \varphi \in  Q^\mathrm{Str}(k + i_1 + \dots + i_k)\]
we divide its outgoing leaves into $k$ outgoing leaves, bounding blocks of $i_1,\dots,i_k$ incoming leaves, and associate to the point
\[ (\varphi,\lambda_1,\dots,\lambda_k) \in Q^\mathrm{Str}_\mathrm{reg}(k;i_1,\dots,i_k) \]
its metric tree but with finite \emph{lengths} $\lambda_1,\dots,\lambda_k$ for each of the outgoing $k$ leaves, modulo an overall shift. In other words, we regularize the diagram by cutting $k$ infinite edges at distances $\lambda_1,\dots,\lambda_k$ from the internal vertex to which they are attached. As a consequence of \cref{thm:homeo}, we immediately have:
\begin{corollary}
    For any tuple $(k;i_1,\dots,i_k)$, there is a homeomorphism
    \[ Q^\mathrm{Str}_\mathrm{reg}(k;i_1,\dots,i_k) \cong \RR^{2k+i_1+ \dots + i_k -4}. \]
\end{corollary}

\noindent Let $(\Gamma,g)$ be the metric graph corresponding to $\varphi$, and let us call the $\Lambda \subset \Gamma$ the subset of these $k$ cutoff points. By the discussion above, the data $(\Gamma,g,\lambda)$ is equivalent to the data of $(\varphi,\lambda_1,\dots,\lambda_k)$. This data gives a \emph{directed graph structure} to the corresponding critical graph, given by taking the gradient of the distance to $\Lambda$ along the graph. There might be some points in the interior of edges of $\Gamma$ at which this distance function attains a local maximum and is not differentiable; these points become vertices with two outgoing edges. There will generically be $(k-1)$ of these points.

In conclusion, the regularized moduli space $Q^\mathrm{Str}_\mathrm{reg}(k;i_1,\dots,i_k)$ is homeomorphic to an open ball of dimension $(2k + i_1 + \dots + i_k -4)$, and has a cellular decomposition into cells labeled by directed planar graphs $\Gamma$ without sinks and possibly with bivalent vertices that have two outgoing edges. These are \textit{non-compact cells}, and we distinguish different parts of their boundaries:
\begin{enumerate}
    \item boundaries ``at finite distance'' where some of the internal edge lengths go to zero, and
    \item boundaries ``at infinity'', where some of the internal edge lengths go to zero and/or some of the lengths of some of the $k$ outgoing edges go to zero or infinity.
\end{enumerate}
The cellular chain differential we consider is then given by
\[ \delta(c_\Gamma) = \sum_{c' \in \text{finite-distance boundary}} \pm c' \] 
that is, given by a sum over cells labeled by graphs obtained by contracting an internal edge of $\Gamma$. These are precisely the same cells and formula for their differentials of the corresponding assocoipahedron, with only a difference in signs and reversed degrees. Checking compatibility of sign conventions gives the following theorem.
\begin{theorem}
    For any tuple $(k;i_1,\dots,i_k)$, the non-compact cell complex $Q^\mathrm{Str}_\mathrm{reg}(k;i_1,\dots,i_k)$ is the dual cell complex of the assocoipahedron $Z_{(k;i_1,\dots,i_k)}$. As a consequence, we have an isomorphism
    \[ (C^{2k+i_1+\dots+i_k-4-*}(Q^\mathrm{Str}_\mathrm{reg}(k;i_1,\dots,i_k),\ZZ), \delta) \cong (C_*(Z_{(k;i_1,\dots,i_k)},\ZZ),d) \]
    between the cellular cochain complex of the regularized moduli space and the cellular chain complex of the assocoipahedron.
\end{theorem}
\noindent This identification gives another proof of the result of Poirier and Tradler.
\begin{corollary}
    The assocoipahedron is contractible, and therefore the dioperad $\scrV^{(n)}$ is Koszul.
\end{corollary}

\section{Koszulity of the dioperad $\scrY^{(n)}$}\label{KoszulityY}
\noindent We now use the same strategy as in the previous section to give a geometric interpretation of the bar complex in the setting of the dioperad $\scrY^{(n)}$. 

\subsection{Cloven quadratic differentials}
We will now look at a certain subspace of quadratic differentials whose associated metric tree is naturally split into regions. Let
\[ (k;i_1,\dots,i_k) \] 
be a tuple specifying the arrangement of in and out arrows around a disk.
\begin{definition}
    Let $(\varphi,\lambda_1,\dots,\lambda_k)$ be a regularized Strebel differential. A \emph{saddle cut} of $\varphi$ is a regular leaf of its \emph{vertical} foliation $p_v(\varphi)$, whose corresponding point in the critical graph of $\varphi$ is a \emph{local maximum} of the function given by the distance to the cutoff set $\Lambda$.
\end{definition}
\noindent Let us motivate the name ``saddle cut''. From the data of $\varphi$ and the cutoff set $\Lambda$ we have a distance-to-$\Lambda$ function $\mathrm{dist}_\Lambda$ on $\CC$; this is a piecewise smooth function, which restricts to the distance function on the critical graph. The unique intersection between a saddle cut $\gamma$ and the critical graph $\Gamma$ is then by definition a saddle point of the function $\mathrm{dist}_\Lambda$.
\begin{figure}[h!]
    \centering
    \includegraphics[width=0.4\linewidth]{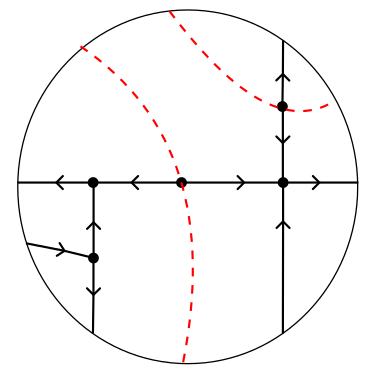}
    \caption{Depiction of the two saddle cuts of a certain regularized Strebel differential $(\varphi,\lambda_1,\dots,\lambda_4)$ with arrangement tuple $(4;0,0,1,1)$. The black lines depict the directed tree associated to the regularized differential; as a tree this is the union of critical horizontal leaves of $\varphi$, and the directions point towards the negative gradient of the distance to the regularizing set. Note that this distance function has two local maxima, corresponding to the saddle cuts in dashed red.}
    \label{fig:differential}
\end{figure}

\begin{lemma}\label{prop:saddle}
    Every saddle cut $\gamma$ is asymptotic to exactly two distinct rays among those with angles
    \[ \frac{2\pi (j+\tfrac{1}{2})}{(k+i_1+\dots+i_k)},\quad  j = 0,\dots, k+i_1+\dots+i_k-1. \]
\end{lemma}
\begin{proof}
    Note that for $|z|$ large enough, we always have $\varphi(z) \sim z^{k+i_1+\dots+i_k-2}dz^2$. The horizontal foliation of a Strebel differential $\varphi$ splits the surface into $(k+i_1+\dots+i_k)$ pieces, each isometric to the upper half-plane. Since $\gamma$ is a regular leaf of the vertical foliation, it must meet exactly two of those pieces, and in each of those under the isometry it maps to a vertical half-ray going to infinity; the result then follows from studying the vertical foliation of $z^{k+i_1+\dots+i_k-2}dz^2$.  
\end{proof}

\noindent Knowing the statement of \cref{prop:saddle}, we can say that the equivalence class $C = [\gamma]$ of a cut $\gamma$ is the unordered pair of indices $\{j_1,j_2\}$ to whose rays it is asymptotic.
\begin{definition}
    Given a collection of pairwise distinct equivalence classes $\mathbf{C} = \{C_1,\dots,C_r\}$, we say that a regularized Strebel differential is \emph{cloven} by $\mathbf{C}$ if it has saddle cuts in those equivalence classes.
\end{definition}
\noindent Simply speaking, $(\varphi,\lambda_1,\dots,\lambda_k)$ is cloven by $\mathbf{C}$ when it has vertical leaves going to the correct pair of regions, each of which intersects the critical graph of $\varphi$ at bivalent vertices. Therefore, for any $\mathbf{C}$ as above, the subspace
\[ \mathrm{ClovQ}_\mathbf{C} \subset Q^\mathrm{Str}_\mathrm{reg}(k;i_1,\dots,i_k) \]
of regularized Strebel differentials that are cloven by $\mathbf{C}$ is a union of cells, each of which labeled by a directed tree that has at least $r$ bivalent vertices. This union of cells is coclosed, in the sense that if $a$ is in $\mathrm{ClovQ}_\mathbf{C}$ and is on the boundary of $b$, then so is $b$. So we can write
\[ \mathrm{Clov}_\mathbf{C} \subset Z_{(k;i_1,\dots,i_k)} \]
for the corresponding union of cells in the assocoipahedron, which will be a cell subcomplex. We now give the main proposition of this section.
\begin{proposition}\label{prop:contractible}
    For any collection of pairwise distinct equivalence classes 
    \[ \mathbf{C} = \{C_1,\dots,C_r\}, \] 
    the space $\mathrm{Clov}_\mathbf{C}$ is contractible if and only if there are $r$ cuts $\gamma_1,\dots,\gamma_r$, one in each equivalence class, that are pairwise disjoint and divide the complex plane into regions such that each contains at least one of the $k$ output directions. In every other case, $\mathrm{Clov}_\mathbf{C}$ is empty.
\end{proposition}
\begin{proof}
    Let us start by proving the contrapositive of the last statement. If $\mathrm{Clov}_\mathbf{C}$ is non-empty, $\mathrm{ClovQ}_\mathbf{C}$ must contain at least one point; taking the $r$ saddle cuts gives us the desired representatives $\gamma_1,\dots,\gamma_r$.

    For the converse, let $\gamma_1,\dots,\gamma_r$ satisfy the condition in the statement. We can easily construct a regularized metric tree whose quadratic differential is cloven by $\mathbf{C}$, say, by picking a single internal vertex for each region, length $2$ for all internal edges and length $1$ for all of the outgoing $k$ leaves. This proves that $\mathrm{ClovQ}_\mathbf{C}$ is not empty.

    Let us now prove contractibility. We first note that the image of $\mathrm{ClovQ}_\mathbf{C}$ under the projection
    \[ \pi \colon Q^\mathrm{Str}_\mathrm{reg}(k;i_1,\dots,i_k) \twoheadrightarrow Q^\mathrm{Str}(k + i_1 + \dots + i_k + 2) \]
    is an open ball, being a product of the metric tree spaces $\mathrm{MetTr}(n_a)$ for some numbers $\{n_a\}$ for $a \in \{1,\dots,r+1\}$, satisfying
    \[ \sum_a n_a = k + i_1 + \dots + i_k + 2r, \]
    where each $n_a$ is the number of leaves of each subtree, once we cleave along the saddle cuts. So $\mathrm{ClovQ}_\mathbf{C}$ is a subspace of
    \[ \pi(\mathrm{ClovQ}_\mathbf{C}) \times (\RR^k)/\RR_{>0} \cong \pi(\mathrm{ClovQ}_\mathbf{C}) \times \Delta^{k-1} \]
    that is, the (trivial) bundle over an open ball with fiber the $(k-1)$-simplex. In each fiber, this subspace is cut out by a certain finite system of linear equations, each one of the form
    \[ \min_{i \in I} (\lambda_i+d_i) < \min_{j \in J} (\lambda_j + d_j) \]
    where $I$ and $J$ are subsets of the outgoing leaf indices $\{1,\dots,k\}$ corresponding to two neighboring regions among the $(k+1)$ regions cut out by the critical graph, and $d_i$ are positive reals, fixed by the metric on the internal edges. This is an intersection of open convex polyhedral subsets of $\Delta^{k-1}$, and therefore an open convex polyhedron. Moreover, this polyhedron in the fiber varies continuously along the base $\pi(\mathrm{ClovQ}_\mathbf{C})$, which proves that $\mathrm{ClovQ}_\mathbf{C}$ is homeomorphic to an open ball, and therefore the corresponding subcomplex $\mathrm{Clov}_\mathbf{C} \subset  Z_{(k;i_1,\dots,i_k)}$ is contractible.
\end{proof}

\noindent The statement above, though rather simple, allows us to characterize the homology type of the subcomplex of the assocoipahedron of all cloven regularized Strebel differentials.
\begin{theorem}
    The subcomplex
    \[ \mathrm{Clov}_{(k;i_1,\dots,i_k)} = \bigcup_\mathbf{C} \mathrm{Clov}_\mathbf{C} \subset Z_{(k;i_1,\dots,i_k)}, \]
    corresponding to any quadratic differentials that are cloven for some saddle cut, has the homology type of a bouquet of some number of $(k-2)$-dimensional spheres.
\end{theorem}
\begin{proof}
    We first note that the inclusion of cellular chain complexes
    \[ C_j(\mathrm{Clov}_{(k;i_1,\dots,i_k)},\ZZ) \subset C_j(Z_{(k;i_1,\dots,i_k)},\ZZ) \]
    is an isomorphism in degrees $0 \le j \le k-2$, since every directed graph with no bivalent vertices has degree at least $(k-1)$. This implies that $H_0(\mathrm{Clov},\ZZ) = H_0(Z_{(k;i_1,\dots,i_k)},\ZZ) = \ZZ$ and $H_j(\mathrm{Clov}_{(k;i_1,\dots,i_k)},\ZZ) =0$ when $j = 1,\dots,k-3$.

    However, the space $\mathrm{Clov}_{(k;i_1,\dots,i_k)}$ is the union of all the spaces $\mathrm{Clov}_{\{C\}}$ for a single equivalence class of saddle cut. These are all contractible spaces, and all their intersections are also either empty or contractible by \cref{prop:contractible}. Therefore, the homology of $C_*(\mathrm{Clov}_{(k;i_1,\dots,i_k)},\ZZ)$ is isomorphic to the homology of a simplicial complex, whose $(r-1)$-simplices are exactly the non-empty intersections of $r$ spaces of the form $\mathrm{Clov}_{\{C\}}$. But there are no such simplices with dimension $\ge (k-1)$, since if $r \ge k$, not all of the $(k+1)$ regions divided by the saddle cuts can have at least one output, and the space $\mathrm{Clov}_{\{C_1,\dots,C_r\}}$ is empty. 
    
    In conclusion, that simplicial complex has dimension $\le (k-2)$ and so the homology of $\mathrm{Clov}_{(k;i_1,\dots,i_k)}$ is supported in degrees $0, \dots, k-2$. Since we established vanishing in degrees $1,\dots,k-3$, the result follows.
\end{proof}

\noindent At the present moment, we do not know of a nice combinatorial formula for the rank of $H_{k-2}(\mathrm{Clov}_{(k;i_1,\dots,i_k)},\ZZ)$, except in the case where $k=2$.
\begin{proposition}\label{prop:rank}
    When $k=2$, we have
    \[ \mathrm{rk}\,H_{0}(\mathrm{Clov}_{(k;i_1,\dots,i_k)},\ZZ) = \prod_{1 \le a < b \le k} (i_a+1)(i_b+1). \] 
\end{proposition}
\begin{proof}
    This is the number of equivalence classes of cuts that keep at least one of the $k$ outputs to each side.
\end{proof}
\begin{figure}
    \centering
    \includegraphics[width=0.5\linewidth]{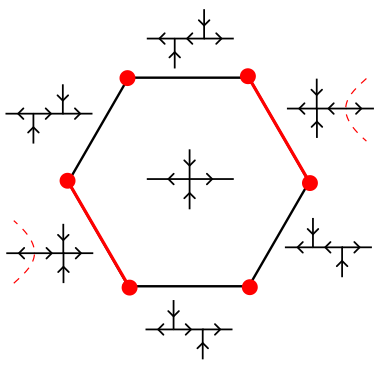}
    \caption{The subcomplex $\mathrm{Clov}_{(2;1,1)}$ in red, inside of the assocoipahedron $Z_{(2;1,1)}$, which is a hexagon. Each cell is labeled by its corresponding directed tree, but we omit the labels of the 0-cells. The dashed lines indicate the saddle cut that cleaves the regularized differentials in that cell.}
    \label{fig:clovenSubcomplex}
\end{figure}

\subsection{Identification of bar complexes}

The following proposition has the same proof as in \cite[Prop.1.15]{emprin2025properadic}, with the difference that the permutations $\sigma$ and $\tau$ of inputs and outputs are constrained to be cyclic in the notion of \cref{sec:planar}.
\begin{proposition}
    There is an embedding of Koszul dual planar dioperads
    \[ (\scrY^{(n)}_{pl})^! \hookrightarrow (\scrV^{(n)}_{pl})^! \]
    whose image is exactly the planar subdioperad spanned by operations of arity different from $(2;0)$.
\end{proposition}
\noindent The embedding above \emph{does not} come from a map of quadratic data, and does not respect the weight grading. It still induces a map of bar complexes, but with a shift in syzygy degree.
\begin{corollary}
    For any tuple $(k;i_1,\dots,i_k)$ with $k \ge 2$ and $i_a \ge 0$, there is an injective map of cochain complexes
    \[ (\mathbf{B}^*_{pl} (\scrY^{(n)}_{pl})^!)_{(k;i_1,\dots,i_k)} \hookrightarrow (\mathbf{B}^{*+k-1}_{pl} (\scrV^{(n)}_{pl})^!)_{(k;i_1,\dots,i_k)} \]
    (note the shift of $k-1$) whose image is spanned precisely by the diagrams without bivalent vertices.
\end{corollary}

\noindent Let us now define
\[ R^* = \coker\left(\mathbf{B}^{*-k+1}_{pl} (\scrY^{(n)}_{pl})^!)_{(k;i_1,\dots,i_k)} \hookrightarrow (\mathbf{B}^{*}_{pl} (\scrV^{(n)}_{pl})^!)_{(k;i_1,\dots,i_k)}\right)\]
to be the cokernel chain complex, spanned by diagrams with bivalent vertices. This immediately implies the following lemma.
\begin{lemma}
    Under the identification between $(\mathbf{B}^{*}_{pl} (\scrV^{(n)}_{pl})^!)_{(k;i_1,\dots,i_k)}$ and the cellular cochain complex of $Z_{(k;i_1,\dots,i_k)}$, the quotient map 
    \[ (\mathbf{B}^{*}_{pl} (\scrV^{(n)}_{pl})^!)_{(k;i_1,\dots,i_k)} \twoheadrightarrow R^* \] 
    identifies $R^*$ with the cellular cochain complex of the subcomplex $\mathrm{Clov} \subset Z_{(k;i_1,\dots,i_k)}$.
\end{lemma}

\begin{proposition}\label{prop:barcomplexKoszul}
    The complex $(\mathbf{B}^*_{pl} (\scrY^{(n)}_{pl})^!)_{(k;i_1,\dots,i_k)}$ has cohomology concentrated in degree zero.
\end{proposition}
\begin{proof}
    For ease of notation, let us write
    \[ Y^* = (\mathbf{B}^*_{pl} (\mathscr{Y}^{(n)}_{pl})^!)_{(k;i_1,\dots,i_k)}, \quad Z^* = (\mathbf{B}^*_{pl} (\mathscr{V}^{(n)}_{pl})^!)_{(k;i_1,\dots,i_k)} \cong C^*(Z_{(k;i_1,\dots,i_k)},\ZZ) \]
    By definition, we have an exact sequence
    \[ Y^{*-k+1} \hookrightarrow Z^* \twoheadrightarrow R^* \]
    and the result then follows from the induced long exact sequence in cohomology, together with the vanishing results above, and contractibility of $Z^*$.
\end{proof}
\noindent We do not know how to compute the rank of $H^0(\mathbf{B}^{*}_{pl} (\scrY^{(n)}_{pl})^!)_{(k;i_1,\dots,i_k)})$ in general, except in the case $k=2$, as a consequence of \cref{prop:rank}.
\begin{corollary}
    When $k=2$, we have
    \[ \mathrm{rk}\,H^0(\mathbf{B}^{*}_{pl} (\scrY^{(n)}_{pl})^!)_{(k;i_1,\dots,i_k)} = \left(\prod_{1 \le a < b \le k} (i_a+1)(i_b+1)\right) - 1.\] 
\end{corollary}
\noindent From \cref{prop:KoszulDiop,prop:barcomplexKoszul} it follows that $(\scrY^{(n)})^! \cong \Phi((\scrY^{(n)}_{pl})^!)$ is a Koszul dioperad; since $\scrY^{(n)}$ is finitely generated, we conclude that $\scrY^{(n)} \cong (\scrY^{(n)})^{!!}$ is also a Koszul dioperad, proving \cref{thm:main}.

\printbibliography

\Addresses

\end{document}